\documentclass[12pt]{article}

\usepackage[margin=1.5in]{geometry}  
\usepackage{graphicx}              
\usepackage{amsmath}               
\usepackage{amsfonts}              
\usepackage{amsthm}                
\usepackage{framed}
\usepackage{tikz}
\usepackage{tkz-berge}

\newtheorem{thm}{Theorem}[section]
\newtheorem{lem}[thm]{Lemma}

\newtheorem{cor}[thm]{Corollary}

\begin{document}

\nocite{*}

\title{\textbf{On the Spectrum of the \\ Generalised Petersen Graphs}}

\author{Adrian W. Dudek\footnote{The author is gracious of the financial support provided by an Australian Postgraduate Award and an ANU Supplementary Scholarship.} \\ 
Mathematical Sciences Institute \\
The Australian National University \\ 
\texttt{adrian.dudek@anu.edu.au}}
\date{}

\maketitle

\begin{abstract}
We show that the gap between the two greatest eigenvalues of the generalised Petersen graphs $P(n,k)$ tends to zero as $n \rightarrow \infty$. Moreover, we provide explicit upper bounds on the size of this gap. It follows that these graphs have poor expansion properties for large values of $n$. We also show that a positive proportion of the eigenvalues of $P(n,k)$ tend to the valency.
\end{abstract}

\section{Introduction}

Let $X$ be a $d$-regular, connected graph on $n$ vertices.  We can list the eigenvalues of the adjacency matrix of $X$ as
$$d=\lambda_1(X) > \lambda_2(X) \geq \cdots \geq \lambda_{n} (X),$$
noting that the greatest eigenvalue is equal to the valency of the graph. The set of these eigenvalues is known as the spectrum of $X$ and the spectral gap of $X$ is defined to be the difference of the two greatest eigenvalues.

The main purpose of this paper is to prove two results about the spectrum of the generalised Petersen graphs. If one considers positive integers $n$ and $k$ such that $n\geq 3$ and $k \leq n/2$, then we can define the generalised Petersen graph $P(n,k)$ to be the graph with vertex set
$$V=\{a_i, b_i : 0 \leq i \leq n-1\}$$
and edge set
$$E=\{a_i a_{i+1}, a_i b_i, b_i b_{i+k} : 0 \leq i \leq n-1 \}.$$
Note that the subscripts must be considered by their image modulo $n$. 

\begin{figure}[hb]
\centering
\includegraphics[width=0.45\linewidth,angle=270]{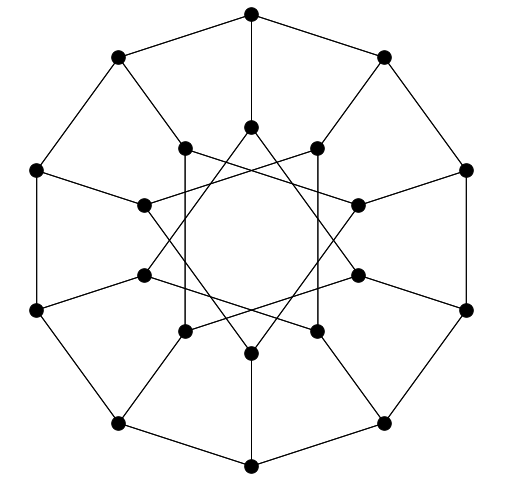}
\caption{The generalised Petersen graph $P(10,3)$.}
\label{fig:plot}
\end{figure}

It is known, as these are connected cubic graphs, that
$$\lambda_1(P(n,k)) = 3 > \lambda_2(P(n,k)).$$ 
Moreover, the entire spectrum of $P(n,k)$ has been given explicitly by Gera and St\v{a}nic\v{a} \cite{gerastanica}; they provide closed form trigonometric expressions for every eigenvalue of $P(n,k)$. This allows them to prove that
$$\lim_{n \rightarrow \infty} \lambda_2(P(n,2)) = 3,$$
that is, the sequence of spectral gaps of $P(n,2)$ tends to zero as $n \rightarrow \infty$. They assert that one could prove this for other fixed values of $k$, but that it would be a difficult problem in general. 

We solve their problem completely in this paper by proving that
$$\lim_{n \rightarrow \infty} \lambda_2(P(n,k)) = 3$$
independently of $k$. More specifically, we prove the following theorem.

\begin{thm} \label{main1}
Let $n \geq 4$ and suppose that $P(n,k)$ is a generalised Petersen graph with spectrum
$$3=\lambda_1 > \lambda_2 \geq \cdots \geq \lambda_{2n}.$$
Then we have the following explicit bound for the spectral gap:
$$\lambda_1 - \lambda_2 < \frac{4 \pi}{[n^{1/2}]-1}.$$

\end{thm}

It follows immediately from the above theorem that
$$\lambda_2(P(n,k)) = 3 + O(n^{-1/2})$$
as $n \rightarrow \infty$, where the implied constant is absolute. This solves the problem posed by Gera and St\v{a}nic\v{a}.

It turns out, however, that we can prove a much stronger type of result. Specifically, we can prove that a positive proportion of the eigenvalues of $P(n,k)$ cluster around the valency as $n$ tends to infinity. To make this notion more clear, we first consider the following theorem of Cioab\u{a} \cite{cioaba} regarding the eigenvalues of Cayley graphs.

\begin{thm}[Cioab\u{a}] \label{cioaba}
For each $\epsilon > 0$ and $k$, there exists a positive constant $C = C(\epsilon,k)$ such that for any Abelian group G and for any symmetric set $S$ of elements of $G$ with $|S| = k$ and $1 \notin S$, the number of eigenvalues $\lambda_i$ of the Cayley graph $X=X(G,S)$ such that $\lambda_i \geq k - \epsilon$ is at least $C \cdot |G|$.
\end{thm}

In short, Cioab\u{a}'s theorem states that a positive proportion of the eigenvalues of a Cayley graph tend towards the valency of the graph. It is known that some generalised Petersen graphs are, in fact, Cayley. Specifically, $P(n,k)$ is a Cayley graph if and only if $k^2 \equiv 1 \text{ mod } n$ (see Nedela and \v{S}kovieria \cite{nedela}). However, the above result is essentially not applicable here; we prove this through the following theorem.

\begin{thm} \label{cayley}
Let $A(N)$ count the number of generalised Petersen graphs $P(n,k)$ up to isomorphism that have $n \leq N$, and let $B(N)$ count those that are also Cayley graphs. Then 
$$\lim_{N \rightarrow \infty} \frac{B(N)}{A(N)} = 0,$$
that is, almost all generalised Petersen graphs are not Cayley graphs.
\end{thm}

We can think of the above theorem as follows. If one were to randomly choose a generalised Petersen graph $P(n,k)$ from all those with $n \leq N$, then the probability that this graph is also a Cayley graph becomes arbitrarily close to zero as $N$ tends to infinity. The above theorem demonstrates the need for an analogue of Cioab\u{a}'s theorem in the setting of generalised Petersen graphs. To this end, we prove the following theorem.

\begin{thm} \label{main2}
For each $\epsilon > 0$, there exists a positive constant $C = C(\epsilon)$ such that for any generalised Petersen graph $P(n,k)$, the number of eigenvalues $\lambda_i$ of $P(n,k)$ such that $\lambda_i \geq 3 - \epsilon$ is at least $C \cdot 2 n$.
\end{thm}

That is, given any generalised Petersen graph, a positive proportion of its eigenvalues tend towards the valency. 

We prove Theorem \ref{main1} and Theorem \ref{main2} by employing Dirichlet's theorem for Diophantine approximation; this is an elementary result in the theory of numbers. We prove Theorem \ref{cayley} by using some standard estimates from analytic number theory. Finally, in Section \ref{expansion}, we interpret our main results in the context of graph expansion.

\section{Proof of Results}

\subsection{Dirichlet's theorem}

To prove Theorem \ref{main1}, we require the following version of Dirichlet's theorem.

\begin{thm}\label{dirichlet}
Given $N$ real numbers $a_1, a_2, \ldots, a_N$, a positive integer $q$, and a positive integer $t_0$, we can find an integer $t$ in the range
$$t_0 \leq t \leq t_0 q^N,$$
and integers $x_1, x_2, \ldots, x_N,$ such that
$$|t a_n - x_n| \leq 1/q$$
for all $1 \leq n \leq N$.
\end{thm}

This is a classic result stating that one can always scale some given set of real numbers by an integer so as to make their fractional parts as small as desired. The tighter the bound imposed on the size of the fractional parts, the larger the range required to guarantee the existence of a scale factor. This is similar to that given in Titchmarsh's text \cite[Ch 8.1]{titchmarsh}, but we have observed that if $t_0$ is an integer, then $t$ will be an integer. We give the proof for completeness, as it is short and relies only on an application of the pigeonhole principle.

\begin{proof}
We consider the $N$-dimensional unit cube with a vertex at the origin and edges along the positive coordinate axes. We divide this cube into $q^N$ equal sub-cubes by partitioning each edge into $q$ equal lengths. Now consider the set of $q^N + 1$ points of the form
$$(ua_1, ua_2, \ldots, ua_N)$$
where $u = 0, t_0, 2 t_0, \ldots, q^N t_0$. By considering the coordinates of each point by their image modulo 1, we have that all $q^N +1$ points lie in the unit cube.  Therefore, by the pigeonhole principle, it follows that at least two of these points are in the same sub-cube. If these two points correspond to $u=a$ and $u =b$ (with $a<b$), then $t = b-a$ satisfies the theorem. Clearly, if $t_0$ is an integer, then $t$ will be an integer.
\end{proof}

To prove Theorem \ref{main2}, we require the following extension of Theorem \ref{dirichlet}.

\begin{thm} \label{dirichlet2}
Given $N$ real numbers $a_1, a_2, \ldots, a_N$, positive integers $m$ and $q$, and a positive integer $t_0$, we can find distinct integers $t_1, t_2, \ldots, t_m$, each of which is in the range
$$t_0 \leq t_i \leq m t_0 q^N,$$
and a set of integers $\{x_{i,n}\}$ such that
$$|t_i a_n - x_{i,n}| \leq 1/q$$
for all $1 \leq n \leq N$ and $1 \leq i \leq m$. 
\end{thm}

\begin{proof}
This is similar to the proof of Theorem \ref{dirichlet}, except that here we give $u$ the values $0, t_0, 2 t_0, \ldots, mq^N t_0$. This gives us $m q^N + 1$ points, and so there must exist a sub-cube containing at least $m+1$ of these points. If we let these points correspond to the values $u  = u_1, u_2, \ldots, u_{m+1}$, then choosing $t_{i} = u_{i+1} - u_1$ proves the theorem. 
\end{proof}

\subsection{Proof of Theorem \ref{main1}}

In the remainder of this paper, we will let $n$ and $k$ be positive integers such that $n \geq 4$ and $1 \leq k \leq n/2$. Corollary 2.5 of Gera and St\v{a}nic\v{a} \cite{gerastanica} provides us with the spectrum of $P(n,k)$; we state their result here.

\begin{thm}
The eigenvalues of $P(n,k)$ are given by
\begin{equation} \label{spectrum}
\cos\bigg( \frac{2 \pi j}{n} \bigg) + \cos \bigg( \frac{2 \pi j k}{n} \bigg) \pm \sqrt{ \bigg( \cos\bigg( \frac{2 \pi j}{n} \bigg) - \cos \bigg( \frac{2 \pi j k}{n} \bigg) \bigg)^2 + 1}
\end{equation}
for $0 \leq j \leq n-1$.
\end{thm}

Clearly, $j=0$ corresponds to two eigenvalues, one of these being $\lambda_1 = 3$. We wish to show that there exists some integer $j$ satisfying $1 \leq j \leq n-1$ such that (\ref{spectrum}) is close to $3$ (where we are using the positive square root). This will ensure that the spectral gap is small.

Using Theorem \ref{dirichlet}, we can show that such an integer $j$ exists. By the periodic nature of $\cos(2 \pi \theta)$, we wish to choose an integer $j$ so that $j/n$ and $jk/n$ are both close to integers. 

As such, we apply Theorem \ref{dirichlet} with  $a_1 = 1/n, a_2 = k/n$, $q = [n^{1/2}]-1$ and $t_0=1$. It follows immediately that there exists an integer $j$ in the range
$$1 \leq j < n,$$
and integers $x_1$ and $x_2$ such that the numbers
$$|j/n - x_1|, |jk/n-x_2|$$
do not exceed $1/([n^{1/2}]-1)$. We now examine what happens to the terms in (\ref{spectrum}) with this choice. Letting $\theta$ denote a real number satisfying $|\theta| \leq 1$, we have that
\begin{eqnarray*}
\cos\bigg( \frac{2 \pi j}{n} \bigg) & = & \cos\bigg( \frac{2 \pi \theta}{[n^{1/2}]-1} \bigg) \\
& > & 1 - \frac{2 \pi}{[n^{1/2}]-1} .
\end{eqnarray*}
The last line of working follows from the fact that $\cos(x) > 1-|x|$. Similiarly, we have
\begin{eqnarray*}
\cos\bigg( \frac{2 \pi j k}{n} \bigg) & > & 1 -  \frac{2 \pi}{[n^{1/2}]-1}.
\end{eqnarray*}

Therefore, using (\ref{spectrum}) and noting that the square-root term is at least 1, it follows that $P(n,k)$ has an eigenvalue $\lambda \neq 3$ such that
$$\lambda > 3 - \frac{4 \pi}{[n^{1/2}]-1} .$$
Clearly, this forms a lower bound for the second largest eigenvalue and so concludes the proof of Theorem \ref{main1}.

\subsection{Proof of Theorem \ref{cayley}}

For this subsection, we employ the notation $f(N) \ll g(N)$ to mean that there exists some $c>0$ such that $f(N) < c g(N)$ for all sufficiently large values of $N$. 

Let $N \geq 3$ be an integer and let $A(N)$ count the number of generalised Petersen graphs $P(n,k)$ up to isomorphism that have $n \leq N$. We want to come up with a lower bound for $A(N)$. For every $n$, we know that $1 \leq k \leq n/2$, and so there are approximately $N^2/4$ choices of pairs $(n,k)$. However, to take isomorphism classes into account, we require the following lemma of Staton and Steimle \cite{statonsteimle}.

\begin{lem}
If $m \geq 5$, then there are exactly $(\varphi(m)+\kappa)/4$ isomorphism classes of $P(m,k)$ where
\begin{enumerate}
\item $k$ is relatively prime to $m$.
\item $\varphi(m)$ is the Euler totient function and denotes the number of integers $n$ such that $1 \leq n \leq m$ and $n$ is relatively prime to $m$.
\item $\kappa$ is the number of solutions to $x^2 \equiv \pm 1 \text{ mod } m$.
\end{enumerate}
\end{lem}

We need a lower bound for $A(N)$, and we can do this by only counting those pairs $(n,k)$ with $\gcd(n,k)=1$. Thus, we have by the above lemma that
$$A(N) > \sum_{n \leq N} \frac{\varphi(n)}{4} \gg N^2$$
as $N \rightarrow \infty$. 

Now, we let $B(N)$ count those generalised Petersen graphs $P(n,k)$ with $n\leq N$, $k \leq n/2$ where $P(n,k)$ is a Cayley graph. We will bound $B(N)$ crudely from above. First, we need the result of Nedela and \v{S}koviera \cite{nedela} that $P(n,k)$ is a Cayley graph if and only if $k^2 \equiv 1 \text{ mod } n$. Therefore, we have that
$$B(N) < \sum_{n \leq N} \sum_{ \substack{k \leq n \\ k^2 \equiv 1 \text{ mod } n}}1$$
where we have not included isomorphism classes or the tighter constraint of $k \leq n/2$ as we only need an upper bound. It follows by the Chinese Remainder Theorem that
$$\sum_{ \substack{k \leq n \\ k^2 \equiv 1 \text{ mod } n}}1 \ll 2^{\omega(n)}$$
where $\omega(n)$ denotes the number of distinct prime divisors of $n$. It is a well known estimate in analytic number theory (though one can see Exercise 4.4.18 of Murty \cite{murty}  for example) that
$$B(N) \ll \sum_{n \leq N} 2^{\omega(n)} \ll N \log N.$$
Collecting our bounds for $A(N)$ and $B(N)$, we have that
$$\frac{B(N)}{A(N)} \ll \frac{\log N}{N}$$
for sufficiently large $N$, and so Theorem \ref{cayley} follows immediately.

\subsection{Proof of Theorem \ref{main2}}

To prove our result, we let $\epsilon > 0$ and apply Theorem \ref{dirichlet2} with $a_1 = 1/n$, $a_2 = k/n$, $q= [4 \pi / \epsilon] +1$, $m = [n/q^2]$ and $t_0 = 1$. To ensure that this all works, we will consider $n$ to be sufficiently large so that $n> q^2$. It follows immediately that there exists $m$ distinct integers $j_1, j_2, \ldots, j_m$ each satisfying 
$$1 \leq j_i < n$$
and integers $x_{i,1},\ldots, x_{m,1}, x_{i,2}, \ldots, x_{m,2}$ such that the numbers
$$|j_i/n - x_{i,1}|, |j_i k/n - x_{i,2}|$$
are all less than $\epsilon/4 \pi$. Therefore, substituting any of these $j_i$ into $(\ref{spectrum})$ and working as in the proof of Theorem \ref{main1} we get an eigenvalue $\lambda$ which satisfies $\lambda> 3 - \epsilon.$ Moreover, as $m = [n/q^2]$, there exists a constant $C=C(\epsilon)$ such that at least $C \cdot 2 n$ eigenvalues of $P(n,k)$ satisfy this bound. This completes the proof.

\subsection{The expanding constant of $P(n,k)$} \label{expansion}

In applications, it can be useful to consider a graph as representing a computer network, where a piece of information is injected into some subset of the vertices and then proceeds to propagate along the edges of the graph at a fixed speed. For such a network to be efficient, we must make the demand that the information spreads quickly throughout the graph no matter which subset of vertices initially contains the information. Thus, one may measure the ability of a graph $X$ to act as a network by its so-called \textit{expanding constant}
$$h(X) = \min_{0 < |F| \leq \frac{|V|}{2}}  \frac{|\partial F|}{|F|}.$$
Here, $F$ ranges over the subsets of the vertices $V$ and $\partial F$ denotes the set of edges which connect a vertex in $F$ to a vertex which is not in $F$. 

One looks to use graphs with large expanding constants in the theory of networks. Importantly though, and possibly for financial reasons, we usually require that the graphs we use are sparse, that is, there are few edges relative to the number of possible edges. As such we consider $d$-regular graphs where each vertex is the endpoint of exactly $d$ edges. 

The isoperimetric inequality for $d$-regular graphs (due to Alon and Milman \cite{alonmilman} and to Dodziuk \cite{dodziuk}) demonstrates the relationship between the expanding constant $h(X)$ and the spectral gap $d-\lambda_2(X)$. We state this here.

\begin{thm}
Let $X$ be a finite, connected, $d$-regular graph without loops. Then

\begin{equation}
\frac{d-\lambda_2(X)}{2} \leq h(X) \leq \sqrt{2d(d-\lambda_2(X))}.
\end{equation}
\end{thm}

The rightmost part of this inequality can be combined with Theorem \ref{main1} to furnish the following result for the expanding constant of the generalised Petersen graphs.

\begin{cor}
Let $n\geq 4$ and $k \leq n/2$. Then
$$h(P(n,k)) < \sqrt{ \frac{24 \pi}{[n^{1/2}]-1}}.$$
\end{cor}

Therefore, the generalised Petersen graphs have poor expansion properties for large values of $n$ i.e. $\lim_{n \rightarrow \infty} h(P(n,k)) = 0$ regardless of how $k$ varies with $n$.

\clearpage

\bibliographystyle{plain}

\bibliography{bibliov2}

\begin{thebibliography}{1}

\bibitem{alonmilman}
N.~Alon and V.~Milman.
\newblock Isoperimetric inequalities for graphs, and superconcentrators.
\newblock {\em J. Comb. Theory}, 38(Ser B):73--88, 1985.

\bibitem{cioaba}
S.~M. Cioab{\u{a}}.
\newblock {C}losed walks and eigenvalues of {A}belian {C}ayley graphs.
\newblock {\em C. R. Math. Acad. Sci. Paris}, 342(9):635--638, 2006.

\bibitem{dodziuk}
J.~Dodziuk.
\newblock Difference equations, isoperimetric inequality and transience of
  certain random walks.
\newblock {\em Transactions of the American Mathematical Society}, 284(2):pp.
  787--794, 1984.

\bibitem{gerastanica}
R.~Gera and P.~St\v{a}nic\v{a}.
\newblock The spectrum of generalized {P}etersen graphs.
\newblock {\em Australasian Journal of Combinatorics}, 49:39--45, 2011.

\bibitem{murty}
M.~R. Murty.
\newblock {\em Problems in Analytic Number Theory}, volume 206 of {\em Graduate
  Texts in Mathematics}.
\newblock Springer, New York, second edition, 2008.

\bibitem{nedela}
R.~Nedela and M.~{\v{S}}koviera.
\newblock Which generalized {P}etersen graphs are {C}ayley graphs?
\newblock {\em Journal of Graph Theory}, 19(1):1--11, 1995.

\bibitem{statonsteimle}
A.~Steimle and W.~Staton.
\newblock The isomorphism classes of the generalized {P}etersen graphs.
\newblock {\em Discrete Mathematics}, 309(1):231--237, 2009.

\bibitem{titchmarsh}
E.~C. Titchmarsh.
\newblock {\em The {T}heory of the {R}iemann {Z}eta-function}.
\newblock Oxford University Press, second edition, 1986.

\end{thebibliography}

\end{document}